\documentclass[reqno,11pt]{amsart}
 
\usepackage{amsmath}
\allowdisplaybreaks[1]
\usepackage{amsthm}  
\usepackage{verbatim}
\usepackage{enumerate} 
\usepackage{mathtools}  
\usepackage{amssymb}
\usepackage{mathrsfs}
\usepackage[top=1.5in, bottom=1.5in, left=0.7in, right=0.7in]{geometry}  
\usepackage[colorlinks]{hyperref}
\hypersetup{
	colorlinks,
	citecolor=blue,
	linkcolor=red
}   
\usepackage{cleveref}
\numberwithin{equation}{section}
\usepackage{bm}
\usepackage{tikz}
\usepackage{tikz-cd}
\usepackage{float}
\usetikzlibrary{arrows.meta}

\newtheorem{theorem}{\textbf{Theorem}}[section]
\newtheorem{theorem*}{\textbf{Theorem}}

\newtheorem{definition}[theorem]{\textbf{Definition}}
\newtheorem{proposition}[theorem]{\textbf{Proposition}}
\newtheorem{lemma}[theorem]{\textbf{Lemma}}

\newtheorem{corollary}[theorem]{\textbf{Corollary}}
\newtheorem{remark}[theorem]{\textbf{Remark}}

\newtheorem{example}[theorem]{\textbf{Example}}
\newtheorem{conjecture}[theorem]{\textbf{Conjecture}}
\newtheorem{definition/proposition}[theorem]{\textbf{Definition/Proposition}}
\newtheorem{innercustomgeneric}{\customgenericname}
\providecommand{\customgenericname}{}
\newcommand{\newcustomtheorem}[2]{%
	\newenvironment{#1}[1]
	{%
		\renewcommand\customgenericname{#2}%
		\renewcommand\theinnercustomgeneric{##1}%
		\innercustomgeneric
	}
	{\endinnercustomgeneric}
}
\newcustomtheorem{customthm}{Theorem}
\newcustomtheorem{customprop}{Proposition}
\newcustomtheorem{customcoro}{Corollary}

\def\N{{\mathbb N}}
\def\R{\mathbb{R}}
\def\Z{{\mathbb Z}}

\def\C{{\mathbb C}}
\def\D{{\mathbb D}}

\def\Q{{\mathbb Q}}
\def\H{{\mathbb H}}

\def\cA{{\mathcal A}}

\def\cL{{\mathcal L}}

\def\hla{{\widehat\lambda}}

\def\hW{{\widehat W}}

\newcommand{\Addresses}{{
		\bigskip
		\footnotesize

        Kainan Guo, \par\nopagebreak
            \textsc{University of Chinese Academy of Sciences, China}\par\nopagebreak
        \textit{E-mail address}: \href{mailto:guokainan20@mails.ucas.ac.cn}{guokainan20@mails.ucas.ac.cn}

	    Zhengyi Zhou, \par\nopagebreak
            \textsc{State Key Laboratory of Mathematical Sciences, Chinese Academy of Sciences;}\par\nopagebreak
	    \textsc{Morningside Center of Mathematics, Chinese Academy of Sciences;}\par\nopagebreak
            \textsc{Academy of Mathematics and Systems Science, Chinese Academy of Sciences, China}\par\nopagebreak
		\textit{E-mail address}: \href{mailto:zhyzhou@amss.ac.cn}{zhyzhou@amss.ac.cn}

}}

\title{The Arnol'd chord conjecture for $ST^*T^3$}
\author{Kainan Guo, Zhengyi Zhou}

\begin{document}
	
\begin{abstract}
We prove the Arnol'd chord conjecture for $ST^*T^3, ST^*(S^1\times S^2),ST^*S^3$ as well as arbitrary connected sums among them.
\end{abstract}
\maketitle
\section{Introduction}\label{s1}
In \cite{Arn86}, Arnol'd stated the following chord conjecture.
\begin{conjecture}[Arnol'd]
    For every closed contact manifold $(Y,\xi)$ and any contact form $\alpha$, every closed Legendrian $\Lambda\subset Y$ carries a Reeb chord.
\end{conjecture}
In comparison to its cousin---the Weinstein conjecture---we only know a handful of contact manifolds for which the chord conjecture holds. Notably, it holds for all contact $3$-manifolds by Hutchings--Taubes \cite{HT13}, for the contact boundary of displaceable Liouville domains by Mohnke \cite{Moh01}, which was generalized to the contact boundary of any Liouville domain with vanishing symplectic cohomology by the second author \cite{Zho22}. There are also cases in which the chord conjecture has been established for certain classes of Legendrian submanifolds in some contact manifolds, e.g.\ for some Legendrian spheres from the surgery perspective by Cieliebak \cite{Cie02}, and for conormals of submanifolds in cosphere bundles by Bro\v{c}i\'c, Cant, and Shelukhin \cite{BCS25}. In this note, we add several classes of contact manifolds for which the chord conjecture holds for all closed Legendrian submanifolds.

\begin{theorem}\label{thm:main}
    The chord conjecture holds for $ST^*T^3$.
\end{theorem}
The proof of \Cref{thm:main} follows the mechanism discovered by Mohnke: if the chord conjecture fails for $\Lambda\subset Y$, then one can find a particular Lagrangian submanifold $L$ of the form $\Lambda\times S^1$ with arbitrarily large $A_{\min}(L,W)$ inside a Liouville filling $W$ of $Y$. Here we define
\begin{equation}\label{def:minimalarea}
        A_{\min}(L,W):=\inf\left\{\int u^*\omega\left|\, u\in\pi_2(W,L),\,\int u^*\omega>0\right.\right\}\in[0,+\infty],
\end{equation}
where the infimum over the empty set is $\infty$. By taking the supremum of $A_{\min}(L,W)$ over all Lagrangian submanifolds $L\subset W$, we obtain $A_{\min}(W)$. Unlike the cases in \cite{Moh01,Zho22}, where $A_{\min}(W)$ is finite, $A_{\min}(DT^*T^3)$ is certainly infinite, as it contains an exact Lagrangian torus. In general, cosphere bundles of closed manifolds form an interesting class of contact manifolds in which to test the chord conjecture, since the boundedness of $A_{\min}$ fails. Although the general results in \cite{Zho22} do not apply directly, the idea of the truncated Viterbo transfer map developed in \cite{Zho22} can be further exploited, with the additional feature that we only need to consider Legendrian surfaces in \Cref{thm:main}. We separate the proof into the low-genus ($\le 1$) and high-genus ($\ge 2$) cases. The low-genus case follows from the geometric construction of capping-off and reduction to the cases treated in \cite{Moh01,Zho22}. The higher-genus case requires the truncated Viterbo transfer map for the BV algebra\footnote{\cite{Zho22} only considered the BV operator, not the whole BV algebra.} structure on symplectic cohomology, as well as some input from the string topology of surfaces.

The ideas behind \Cref{thm:main} also yield the following cases of the chord conjecture.
\begin{theorem}\label{thm:chord}
    The chord conjecture holds for the following pairs $(Y, \Lambda)$:
    \begin{enumerate}
        \item\label{c1} $Y=\partial V$, where $V$ is a Liouville domain that embeds \emph{symplectically} into a Liouville domain $W$ with $SH^*(W;\Z)=0$, and $\Lambda$ is a closed Legendrian submanifold satisfying $H_1(\Lambda;\Q)=0$;
        \item\label{c2} $Y=\partial (DT^*S^1\times V)$, where $V$ is a Liouville domain and $\Lambda$ admits a metric of negative sectional curvature.
    \end{enumerate}
\end{theorem}

Condition \eqref{c1} of \Cref{thm:chord} is satisfied in many cases, e.g.\ (1) $Y=\partial (DT^*S^1\times V)$, since $DT^*S^1\times V$ embeds symplectically into $\D\times V$, and (2) the contact boundary of the complement of a complex hyperplane arrangement in $\C^n$. In particular, by combining \eqref{c1} with condition \eqref{c2} of \Cref{thm:chord}, we see that the chord conjecture holds for $ST^*(S^1\times \Sigma)$ for any surface $\Sigma$ and any Legendrian surface that is neither a torus nor a Klein bottle. More generally, there are a few other cases of cosphere bundles of $3$-manifolds and their modifications for which we can establish the chord conjecture.
\begin{theorem}\label{thm:CotBundle}
    The chord conjecture holds for the following contact $5$-manifolds:
    \begin{enumerate}
        \item\label{cot1} $ST^*S^3$ and $ST^*(S^1\times S^2)$.
        \item\label{cot2} Arbitrary contact connected sums (duplicates allowed) among

        $$\{ST^*T^3,\ ST^*S^3,\ ST^*(S^1\times S^2),\ ST^*(S^1\times \R P^2),\ \partial V \text{ with } SH^*(V;\Z)=0\}.$$
    \end{enumerate} 
\end{theorem}
\begin{remark}
    It is clear that if the chord conjecture holds for $Y$, then it holds for any finite cover of $Y$. As a consequence, for a $3$-dimensional manifold $M$ carrying one of the eight geometries, the chord conjecture holds for $ST^*M$ when $M$ has $S^3$, $E^3$, or $S^2\times E$ geometry. 
\end{remark}

\begin{remark}
    In the forthcoming paper of Bro\v{c}i\'c and Cant \cite{BC}, the authors find another condition on the symplectic cohomology as a BV algebra that guarantees the validity of the chord conjecture for all Legendrians. In particular, they show that the chord conjecture holds for $ST^*(S^{n_1}\times \cdots \times S^{n_k})$ for any $n_i\in \N_+$.
\end{remark}

\subsection*{Acknowledgments}
The authors would like to thank Dylan Cant for helpful discussions and for informing us of their forthcoming paper \cite{BC}. This project is supported by the National Key R\&D Program of China under Grant No.\ 2023YFA1010500 and by the National Natural Science Foundation of China under Grants No.\ 12288201 and 12231010.

\section{Symplectic cohomology}\label{s2}

In this section, we briefly review the construction of symplectic cohomology for Liouville domains and explain the additional algebraic structures carried by symplectic cohomology. We then discuss the Viterbo transfer map and its truncated version. We refer the reader to \cite{CO18, Rit13} for a more complete and detailed treatment.

\subsection{Properties of symplectic cohomology}\label{2.1.1}
\subsubsection{Definition}
Let $(W,\lambda)$ be a Liouville domain, and let $(\hW, \hla)=(W\cup \partial W\times[1,\infty)_r,\ e^{r}\lambda)$ be its completion, where $\partial W$ is identified with $\partial W\times \{1\}$ and the Liouville vector field provides a collar neighborhood $(0,1]\times \partial W$ of $\partial W$, which is used to define the smooth structure on the completion. Let $H:S^1\times \hW\to\R$ be a time-dependent Hamiltonian on $\hW$. For a loop $\gamma: S^1\to\hW$, the action functional is given by
\begin{equation}
    \cA_H(\gamma)=-\int \gamma^*\hla+\int_{0}^{1}H(t,\gamma(t))\,\mathrm{d}t,
\end{equation}
where $\mathrm{d}\hla(\cdot,X_H)=\mathrm{d}H$. The critical points of the action functional correspond to the $1$-periodic Hamiltonian orbits. We call a Hamiltonian \emph{admissible} if it is a $C^2$-small nondegenerate perturbation of a Hamiltonian $K_a$ which is zero on $W$ and linear at infinity with slope $a$, where $a$ is not the period of any Reeb orbit on $\partial W$; see Figure~\ref{fig:K_a}. We denote such an admissible Hamiltonian by $H_a$.\par
\begin{figure}
    \centering
    \begin{tikzpicture}[>=Stealth, scale=1]
    
    \draw[->, thick] (0,0) -- (4.5,0) node[below] {$r$};
    \draw[->, thick] (0,0) -- (0,2) node[left] {$H_a$};

    \draw[ultra thick] (0,0) -- (1.8,0);
    
    \draw[ultra thick] (1.8,0) .. controls (2.2,0) and (2.4,0.1) .. (2.6,0.3);
    
    \draw[ultra thick] (2.6,0.3) -- (3.8,1.4);

    \draw (1.8, 0.05) -- (1.8, -0.05) node[below] {$1-\delta$};
    \draw (2.6, 0.05) -- (2.6, -0.05) node[below] {$1$};

    \draw[<-] (0.8, -0.1) -- (0.5, -0.3) node[below, align=center] {constant \\ orbits};
    
    \draw[<-] (2.3, 0.2) -- (2.1, 1.2) node[above, align=center] {non-constant \\ orbits};
    
    \node[right] at (3.4, 1.5) {slope $= a$};

\end{tikzpicture}
    \caption{Definition of $H_a$.}
    \label{fig:K_a}
\end{figure}
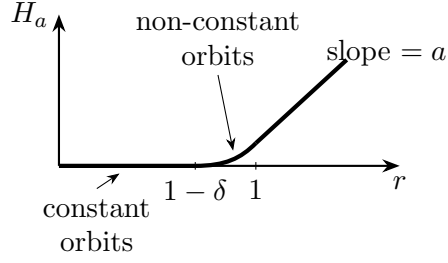
Fix a suitable almost complex structure $J$. The Floer cochain complex $CF^*(H;R)$ of a Hamiltonian $H$ with $R$-coefficients is generated by the critical points of $\cA_H$, and the differential is given by counting rigid Floer cylinders asymptotic to Hamiltonian orbits in both the positive and the negative directions. In particular, under the assumption that $\lambda|_{\partial W}$ has no degenerate Reeb orbit of period at most $a$ and that $H_a$ is in the special form of \cite{CO18}, $CF^*(H_a;R)$ is generated by constant orbits in the interior of $W$, corresponding to the Morse cochain complex of $W$, and by non-constant orbits in the collar neighborhood of $\partial W$, which come in pairs corresponding to Reeb orbits on $\partial W$ of period at most $a$. We denote the Floer cohomology of $H_a$ by $HF^*(H_a;R)$. 

Inside $CF^*(H_a;R)$, there is a subcomplex $CF_0^*(H_a;R)$ generated by those constant orbits on $W$, whose cohomology is isomorphic to $H^*(W;R)$. The quotient complex is denoted by $CF^*_+(H_a;R)$, and its cohomology is denoted by $HF^*_+(H_a;R)$. These groups fit into a long exact sequence in the usual way.

\begin{proposition}[See for example {\cite[\S 3e]{Sei08}}]
The truncated symplectic cohomology groups
\[
SH^*_{<a}(W;R):=HF^*(H_a;R),\qquad SH^*_{+,<a}(W;R):=HF^*_+(H_a;R)
\]
are well-defined, i.e.\ independent of the choice of $H_a$, $J$, etc.
\end{proposition}
There is a continuation map $c:SH^*_{<a}(W;R)\to SH^*_{<b}(W;R)$ for $a\leq b$, which on the cochain level is compatible with the decomposition into $CF_0$ and $CF_+$. The full (positive) symplectic cohomology is the direct limit
\[
SH^*(W;R):=\varinjlim_{a\to\infty} SH^*_{<a}(W;R), \qquad SH^*_+(W;R):=\varinjlim_{a\to\infty} SH^*_{+,<a}(W;R),
\]
and we obtain the long exact sequence
\begin{equation}\label{splitting}
    \cdots \to H^*(W;R) \to SH^*(W;R) \to SH_+^*(W;R) \to H^{*+1}(W;R) \to \cdots .
\end{equation}
\begin{remark}
    $SH^*(W)$ admits a canonical $\Z/2$-grading given by $\dim_{\C}W-\mu_{\mathrm{CZ}}$, where $\mu_{\mathrm{CZ}}$ is the Conley--Zehnder index. This grading can be lifted to a $\Z$-grading after choosing a trivialization of $\det_{\C}TW$. In particular, $SH^*(T^*Q)$ admits a canonical $\Z$-grading \cite[\S 9.4.5]{Abo15}. Symplectic cohomology also admits a grading by $\mathrm{Conj}(\pi_1(W))$, coming from homotopy classes of Hamiltonian orbits, since the asymptotics of a Floer cylinder lie in the same homotopy class.
\end{remark}

\subsubsection{Pair-of-pants product and BV operator}
Symplectic cohomology is equipped with TQFT-type algebraic structures. The standard treatment of these structures can be found, e.g., in \cite{Abo15,Rit13}. From now on, we assume that the Liouville domain $(W,\lambda)$ satisfies $c_1(TW,J)=0$, so that $SH^*(W;R)$ admits a $\Z$-grading once we fix a trivialization of $\det_{\C} TW$.\par
The pair-of-pants product
\begin{equation}
    \psi_P:SH^i_{<A}(W;R)\otimes SH^j_{<B}(W;R)\to SH^{i+j}_{<A+B}(W;R)
\end{equation}
is defined by counting rigid curves solving suitable Floer equations on the pair-of-pants surface. It is compatible with continuation maps, which yields a product on $SH^*(W;R)$.\par
The pair-of-pants product $\psi_P$ is:
\begin{itemize}
    \item graded-commutative: $\psi_P(x,y)=(-1)^{|x|\,|y|}\psi_P(y,x)$ \cite[\S 6.2]{Rit13};
    \item associative: $\psi_P(x,\psi_P(y,z))=\psi_P(\psi_P(x,y),z)$ \cite[\S 6.2]{Rit13};
    \item unital: there exists $e\in SH^*(W)$ (which need not be nonzero) such that $\psi_P(e,x)=\psi_P(x,e)=x$ \cite[\S 6.3]{Rit09}.
\end{itemize}
Furthermore, the pair-of-pants product $\psi_P$ on $SH^*(W)$ is compatible with the usual cup product on $H^*(W)$, so that $H^*(W;R) \to SH^*(W;R)$ is a unital ring homomorphism.\par

By counting Floer-type cylinders with an asymptotic marker that moves in an $S^1$ family at the negative end, we associate to symplectic cohomology a degree $-1$ map $\Delta:SH^*(W;R)\to SH^{*-1}(W;R)$, called the Batalin--Vilkovisky operator. In string topology, the natural $S^1$-action
\begin{equation}
\begin{split}\label{S1act}
    \rho:S^1\times \cL Q&\to \cL Q,\\
    (t,\gamma(s))&\mapsto \gamma(t+s),
\end{split}
\end{equation}
induces the BV operator of degree $1$ on loop homology:
\begin{equation}
    \begin{split}
        \Delta:H_*(\cL Q;R)&\to H_{*+1}(\cL Q;R),\\
        x&\mapsto \rho_*([S^1]\otimes x),
    \end{split}
\end{equation}
where $[S^1]\in H_1(S^1)$ is the fundamental class of $S^1$. These two operators coincide via the Viterbo isomorphism when $W=T^*Q$; see Theorem~\ref{Vitiso}.
\begin{proposition}[\cite{Abo15}]\label{prop:BV}The BV operator has the following properties:
\begin{enumerate}
    \item $\Delta$ does not decrease the symplectic action. In particular, it is defined on the filtered symplectic cohomology, $\Delta: SH^*_{<a}(W;R)\to SH^*_{<a}(W;R)$.
    \item $\Delta$ preserves the homotopy classes of Hamiltonian orbits.
    \item $\Delta^2=0$ on symplectic cohomology.
    \item $\Delta=0$ on $SH^*_{<\epsilon}(W;R)=H^*(W;R)$ for $0\le \epsilon \ll 1$.
\end{enumerate}
\end{proposition}
The BV operator has an additional compatibility with the product structure, which leads to the following BV algebra structure on symplectic cohomology.
\begin{definition}[\cite{GP21}]
    A \emph{Batalin--Vilkovisky algebra} is a graded vector space $A$ equipped with the following structure:
    \begin{itemize}
        \item A graded-commutative product: $a\cdot b=(-1)^{|a|\,|b|}b\cdot a$.
        \item A Gerstenhaber bracket: a Lie bracket $[\cdot,\cdot]$ of degree $-1$ which is graded-commutative and satisfies the Jacobi identity and the Leibniz rule.
        \item A BV operator: a linear operator $\Delta:A\to A$ of degree $-1$ satisfying $\Delta^2=0$ and $[a,b]=\Delta(a\cdot b)-(\Delta a)\cdot b- (-1)^{|a|}a\cdot (\Delta b)$.
    \end{itemize}
\end{definition}
See \cite[Chapter 10]{Abo15} for the explicit construction of the BV algebra structure on symplectic cohomology and \cite{CS99} for the BV structure on loop homology.

\subsection{Viterbo isomorphism}
In \cite{Vit99}, Viterbo proved that for an oriented closed manifold $Q$, the symplectic cohomology of $T^*Q$ is isomorphic to the homology of the free loop space $\cL Q$, using generating function theory. Later, this isomorphism was reformulated by Abbondandolo--Schwarz \cite{AS06} and Salamon--Weber \cite{SW06} using different approaches. The necessity of twisting by a local system was noticed by Seidel and worked out in detail by Kragh \cite{Kra18}. The local coefficient system is trivial if and only if the second Stiefel--Whitney class vanishes on every torus in the cotangent bundle. The reason for this phenomenon is that the orientability of the moduli spaces being counted is determined by several characteristic classes; see \cite{AS14,Abo15}. To avoid these issues, we mainly use the Viterbo isomorphism with $\Z/2$ coefficients.
\begin{theorem}[\cite{AS14, Abo15,Vit99}]\label{Vitiso}
There is a local coefficient system $\eta$ on the free loop space such that
\[
SH^*(T^*Q;\Z) \cong H_{n-*}(\mathcal{L}Q;\eta).
\]
Moreover, if $Q$ is orientable and the second Stiefel--Whitney class vanishes on tori in $T^*Q$, then $\eta$ is trivial, in which case
\[
SH^*(T^*Q;\Z) \cong H_{n-*}(\mathcal{L}Q;\Z).
\]
In all cases,
\[
SH^*(T^*Q;\Z/2) \cong H_{n-*}(\mathcal{L}Q;\Z/2).
\]
In particular, the isomorphisms above are isomorphisms of BV algebras.
\end{theorem}

\subsection{Viterbo transfer}
Consider an inclusion map $i:V^n\hookrightarrow W^n$ of manifolds, which induces a restriction map in cohomology, $i^*:H^*(W)\to H^*(V)$. If $(V,\lambda_V=\lambda_W|_V)$ is a Liouville subdomain of the Liouville domain $(W,\lambda_W)$, there is an analogous restriction map, called the \emph{Viterbo transfer map}, $i^*:SH^*(W)\to SH^*(V)$ \cite{Vit99}, which preserves the BV algebra structure. In particular, if $L\subset W$ is an exact Lagrangian, then a Weinstein neighborhood of $L$ is an exact subdomain in $W$, yielding a BV algebra map $SH^*(W)\to SH^*(T^*L)$.\par

In general, when $L$ is not exact, Weinstein neighborhoods of $L$ are not exact subdomains, which breaks the usual Viterbo transfer. Specifically, the restriction of the Liouville form $\lambda$ to $L$ is then a closed but non-exact $1$-form. However, a deformed Viterbo transfer should still exist \cite[Remark 4.2]{Zho22}, which leads to the truncated Viterbo transfer considered in \cite[\S 3]{Zho22}. To state the truncated Viterbo transfer, we first recall \textbf{a modified version of the minimal symplectic area} of $L$:
$$A'_{\min}(L,W):=\inf \left\{\langle \gamma,\, \lambda|_L \rangle\,\left|\, \gamma\in H_1(L),\ \langle \gamma,\lambda|_L\rangle>0 \right.\right\}.$$
The difference with \eqref{def:minimalarea} is that, by Stokes' theorem, \eqref{def:minimalarea} only uses loops $\gamma$ that are contractible in $W$. This quantity controls the minimal energy level at which the deformation can be observed, and serves as the threshold for the truncated Viterbo transfer as follows.

\begin{proposition}[{\cite[Proposition 3.1]{Zho22}}]\label{prop:truncate}
Let $(W,\lambda)$ and $L$ be as above, and assume that $A'_{\min}(L,W) > 2A > 0$. Then there is a Viterbo transfer map $SH^*_{<A}(W) \to SH^*(T^*L)$ such that the following diagram of long exact sequences commutes:
\[
\begin{tikzcd}[column sep=small]
\dots \arrow[r] & H^*(W) \arrow[r] \arrow[d] & SH^{*}_{<A}(W) \arrow[r] \arrow[d] & SH^{*}_{+,<A}(W) \arrow[r] \arrow[d] & H^{*+1}(W) \arrow[r] \arrow[d] & \dots \\
\dots \arrow[r] & H^*(T^*L) \arrow[r] & SH^*(T^*L) \arrow[r] & SH^*_+(T^*L) \arrow[r] & H^{*+1}(T^*L) \arrow[r] & \dots
\end{tikzcd}
\]
This Viterbo transfer map is compatible with the BV operator, and compatible with the product
$$\psi_P: SH_{<A}^*(W)\otimes SH_{<A}^*(W)\to SH_{<2A}^*(W)$$
whenever $A'_{\min}(L,W)>4A$.
\end{proposition}
The difference with \cite[Proposition 3.1]{Zho22} is that here we consider $SH^*(W)$ generated by \emph{all} orbits, rather than only contractible orbits as in \cite[Proposition 3.1]{Zho22}. Accordingly, we need the refined bound using $A'_{\min}(L,W)$ instead of $A_{\min}(L,W)$. The proof is identical to that of \cite[Proposition 3.1]{Zho22}.

\section{Loop homology}\label{s3}

To derive the contradiction outlined in Section~\ref{s1}, we now discuss the string topology of $T^3$ and of closed surfaces. We first introduce the following specific classes.
\begin{definition}[\cite{GP21,Li19,Sei13,SS12,Zho22'}]
    Let $(A,\Delta)$ be a (graded) BV algebra over a ring $R$, and let $e_0,e_1 \in A$ be elements satisfying $\Delta(e_1)=e_0$. We say that $(e_0,e_1)$ is
\begin{itemize}
    \item an \emph{$R$-dilation} if $e_0=1\in A$;
    \item an \emph{$R$-quasi-dilation} if $e_0$ is invertible in $A$.
\end{itemize}
\end{definition}
\begin{remark}
    Note that an $R$-dilation in $SH^*(V;R)$ must be represented by a non-constant contractible orbit, by Proposition~\ref{prop:BV}.
\end{remark}
\begin{proposition}\label{prop:dilation}(Quasi-)dilations (with the same $R$) are preserved under the following constructions:
    \begin{enumerate}
        \item If $(V,\lambda_V)$ is a Liouville subdomain of $(W,\lambda_W)$ and $SH^*(W)$ carries a (quasi-)dilation, then so does $SH^*(V)$.
        \item If one of the Liouville domains $V$ and $W$ admits a (quasi-)dilation, then so does $V\times W$.
    \end{enumerate}
\end{proposition}
\begin{proof}
    The first claim follows from the fact that the Viterbo transfer is compatible with the BV algebra structure. The second follows from the K\"unneth formula for symplectic cohomology established in \cite{Oan06}: if $SH^*(V)$ admits a quasi-dilation $x$, then $x\otimes 1$ is a quasi-dilation of $SH^*(V\times W)$.
\end{proof}
Now we consider several examples coming from string topology. Denote $\H^*(Q;R):=H_{n-*}(\cL Q;R)$, where $n=\dim Q$.
\begin{example}[\cite{CJY03,MG06}]\label{BVS1} 
    \begin{equation}\label{S1}
        \H^*(S^1;R)\cong R[t,t^{-1}]\otimes \wedge[a],
    \end{equation}
    where $t$ and $a$ are formal variables with $\deg t=0$ and $\deg a=1$, and $\wedge[a]$ is the exterior algebra over $R$ generated by $a$. Moreover, the BV operator is given by
    \begin{equation}
            \Delta(t^i\otimes a)= it^i\otimes 1,\qquad \Delta(t^i\otimes 1)=0,
    \end{equation}
    and the loop product is given by
    \begin{equation}
            (t^i\otimes a)\circ (t^j\otimes a) =0,\quad (t^i\otimes a)\circ (t^j\otimes 1) = t^{i+j}\otimes a, \quad (t^i\otimes 1)\circ(t^j\otimes 1)=t^{i+j}\otimes 1.
    \end{equation}
    In other words, $\H^*(S^1)$ admits an $R$-quasi-dilation for any $R$.
\end{example}
\begin{remark}\label{BVT2}
    The previous example implies that $\H^*(T^n;R)\cong \H^*(S^1;R)^{\otimes n}$ admits quasi-dilations but no dilations. More generally, on a manifold with non-positive curvature, the only contractible geodesics are the constant ones \cite{Kli78,Kli82}. In particular, there is no dilation on the loop homology. Hence, the Klein bottle and the hyperbolic surfaces do not admit dilations over any $R$.
\end{remark}
\begin{example}[\cite{MG06,SS12}]\label{BVS2}
    $\H^*(S^n)$ admits $R$-dilations for any $R$ if $n\geq 3$, and $\H^*(S^2)$ admits $R$-dilations if and only if $2$ is invertible in $R$. 
\end{example}
\begin{example}\label{BVRP2}
    $SH^*(T^*\R P^2;\Q)$ admits a dilation. This can be computed explicitly using \cite[Theorem 5.1]{Zho22}, since $T^*\R P^2$ is the complement of a degree $2$ curve in $\C\mathbb{P}^2$, or derived from \cite[Example 6.4]{SS12}. Moreover, the computation in \cite[Theorem 5.1]{Zho22} shows that $SH^*(T^*\R P^2;\Z/2)=\H^*(\R P^2;\Z/2)$ does not admit a dilation.
\end{example}

For a compact manifold $M$ of negative sectional curvature, $M$ is the classifying space of $\pi_1(M)$ by the Cartan--Hadamard theorem. The loop homology of the classifying space can then be computed algebraically.
\begin{theorem}[{\cite[Ch.~7]{Lod98}}]\label{HH=HL}
    For any discrete group $G$ and any commutative ring $R$ with unit, there is a canonical isomorphism
\[
\mathrm{HH}_*(R[G]) \cong H_*(\cL BG;R).
\]
\end{theorem}
\begin{theorem}[\cite{Bur85}]\label{decom}
For any discrete group $G$, there is a canonical isomorphism
\[
\mathrm{HH}_*(R[G]) \cong \bigoplus_{\langle z \rangle \in \langle G \rangle} H_*(BG_z;R),
\]
where $G_z$ is the centralizer of $z$, $\langle G \rangle$ is the set of conjugacy classes of $G$, and $\langle z \rangle$ denotes the conjugacy class of $z$.
\end{theorem}
Thus, the loop homology of the classifying space is determined by the conjugacy classes of the fundamental group. We will use the Preissmann theorem \cite[Ch.~12]{Car92}.
\begin{theorem}[Preissmann]
    If $M$ is a compact Riemannian manifold of negative sectional curvature, then any non-trivial abelian subgroup of $\pi_1(M)$ is isomorphic to $\Z$.
\end{theorem}
\begin{proposition}\label{prop:negaCurv}
    The loop homology of a closed manifold $M$ of dimension $n\geq 2$ and negative sectional curvature is
    \[
    \H^*(M)\cong H_{n-*}(M)\oplus \bigoplus_{\langle\gamma\rangle\neq 1}H_{n-*}(S^1),
    \] 
    where $\langle\gamma\rangle$ ranges over the non-trivial conjugacy classes of $\pi_1(M)$. Moreover, $\H^*(M)$ does not admit a quasi-dilation.
\end{proposition}
\begin{proof}
    For the trivial conjugacy class $\langle 1\rangle$, we have $H_*(G_{1})=H_*(G)=H_*(M)$. For a non-trivial conjugacy class $\langle\gamma\rangle$, the powers $\gamma^k$ lie in $G_{\gamma}$. If there is an element $g\in G$ such that $g\gamma g^{-1}=\gamma$ and $g$ is not a power of $\gamma$, then $g$ and $\gamma$ generate an abelian subgroup containing $\Z\oplus\Z$ in $\pi_1(M)$, contradicting Preissmann's theorem. Hence $G_{\gamma}\cong\Z$ and $H_*(G_{\gamma})=H_*(\Z)=H_*(S^1)$. Therefore $\H^*(M)=\mathrm{HH}_{n-*}(R[G])=H_{n-*}(M)\oplus\bigoplus_{\langle\gamma\rangle\neq 1}H_{n-*}(S^1)$ by Theorems~\ref{HH=HL} and~\ref{decom}.

    Note that $\H^0(M)$ is generated by the fundamental class of $M$ via constant loops; everything else has cohomological degree at least $1$. By degree considerations, an invertible element of $\H^*(M)$ must, up to a nonzero scalar, be of the form $1+a$ where $a$ has higher degree. If a quasi-dilation exists, then by degree reasons we can find $e_1\in \H^1(M)$ such that $\Delta(e_1)=1$. When $n\ge 3$, such an $e_1$ must come from $H_{n-1}(M)$, contradicting $\Delta=0$ on the homology of constant loops, i.e.\ on the $H_*(M)$ component of $\H^*(M)$. When $n=2$, $e_1$ could come from either $H_1(M)$ or $H_1(S^1)$ by degree reasons; this contradicts both $\Delta=0$ on constant loops and the fact that $\Delta$ preserves the homotopy class of loops.
\end{proof}
\begin{remark}
    Proposition~\ref{prop:negaCurv} can also be derived from Morse theory on the loop space of $M$, since the existence of negative curvature guarantees that the Morse index of every non-constant closed geodesic orbit is zero \cite[Lemma 2.2]{CM18}. The loop homology generated by all non-constant closed geodesic orbits is therefore concentrated in two degrees, as above.
\end{remark}

\section{Proof of \Cref{thm:main}}\label{s4}
\subsection{Chord conjecture for $ST^*T^3$}
We first recall Mohnke's proof \cite{Moh01} in order to motivate ours. It is convenient to replace $A_{\min}(L,W)$ and $A_{\min}(W)$ by $A'_{\min}(L,W)$ and $A'_{\min}(W)$ as in Proposition~\ref{prop:truncate}; that is, for a Lagrangian $L$ in a Liouville domain $(W,\lambda)$,

$$A'_{\min}(L,W):=\inf \left\{\langle \gamma,\, \lambda|_L \rangle\,\left|\, \gamma\in H_1(L),\ \langle \gamma,\lambda|_L\rangle>0 \right.\right\},\qquad A'_{\min}(W)=\sup_{L\subset W} A'_{\min}(L,W).$$
When $\gamma$ is represented by a loop that is contractible in $W$, then $\langle \gamma,\lambda|_L\rangle=\int u^*\omega$ for any extending disk $u:D^2\to W$ with $u|_{\partial D^2}=\gamma$. From this it follows that $A'_{\min}\le A_{\min}$. Hence Mohnke's argument is not affected, while the modified $A'_{\min}$ is easier to work with.

For any $(\dim_{\C}W-1)$-dimensional closed manifold $\Lambda$ and an embedding $\iota:\Lambda \to W$, we use $L_{\Lambda,\iota}$ to denote a Lagrangian of the form $\Lambda\times S^1$ such that the $S^1$-factor is contractible in $W$ and $\Lambda \times \{*\}$ is isotopic to the embedding $\iota$. We use $\cL_{\Lambda,\iota}$ to denote the set of all such Lagrangians. We then define
$$A'_{\min}(\cL_{\Lambda,\iota},W) := \sup \left\{A'_{\min}(L,W)\,\big|\, L\in \cL_{\Lambda,\iota}\right\}.$$

\begin{theorem}[\cite{Moh01}]\label{thm:Mohnke}
    Let $W$ be a Liouville domain. If 
    \[
    A'_{\min}(\cL_{\Lambda,\iota},\, W)<+\infty,
    \]
    then the chord conjecture holds for $(\partial W,\ker\lambda|_{\partial W})$ for every Legendrian $\Lambda \subset \partial W$ such that the inclusion $\Lambda \hookrightarrow W$ is isotopic to $\iota$.
\end{theorem}
\begin{proof}\label{pro:Mohnke}
Let $\alpha$ be a contact form on $\partial W$, which need not coincide with $\lambda|_{\partial W}$. We have $\alpha = f\lambda|_{\partial W}$ for some positive function $f$ on $\partial W$. We can find $r_0 > 0$ such that $0<r_0f \leq 1$; that is, the contact hypersurface $\{r = r_0f\}$ is contained in the collar neighborhood $(0,1]\times \partial W$ of $\partial W$, and hence in $W$. Then $\lambda$ restricts to $r_0\alpha$ on this hypersurface. Since the Reeb flows of $r_0\alpha$ and $\alpha$ agree up to rescaling, the chord conjecture holds for $\alpha$ if and only if it holds for $r_0\alpha$. Let $W'$ denote the exact subdomain bounded by $\{r = r_0f\}$. It is straightforward to check that
\[
A'_{\min}(\cL_{\Lambda,\iota},\, W') \leq A'_{\min}(\cL_{\Lambda,\iota},\, W) < \infty.
\]
Let $\phi_t$ denote the Reeb flow on $\partial W'$. Assume that the Legendrian $\Lambda$ has no Reeb chord. Then 
\begin{equation}
    \begin{split}
        \Phi : \Lambda \times [\epsilon, 1] \times \R &\to \bigl(\partial W' \times [\epsilon, 1]_{r'},\, \mathrm{d}(r'\lambda|_{\partial W'})\bigr), \\
        (x, s, t) &\mapsto (\phi_t(x), s)
    \end{split}
\end{equation}
is an embedding, where $r'$ is the cylindrical coordinate for the symplectization of $\partial W'$. Any embedded loop $\gamma \subset [\epsilon, 1]_s \times \mathbb{R}_t$ gives rise to a Lagrangian submanifold $\Phi(\Lambda \times \gamma) \subset \partial W' \times [\epsilon, 1] \subset W'$. Since $\lambda|_{\Lambda} = 0$ and $\Phi^*(\mathrm{d}(r'\lambda|_{\partial W'})) = \mathrm{d} s \wedge \mathrm{d} t$, the quantity $A'_{\min}(\Phi(\Lambda \times \gamma), W')$ equals the area enclosed by $\gamma$, which can be arbitrarily large. This contradicts the assumption that $A'_{\min}(\cL_{\Lambda,\iota},\, W')<\infty$.
\end{proof}

We begin with the capping-off construction and prove the chord conjecture for $ST^*T^3$ in the case of Legendrian surfaces of genus $\leq 1$. We write $(DT^*S^1,\lambda_{DT^*S^1}) = ([-1,1]_p\times S^1_{\theta},\, p\,\mathrm{d}\theta)$.

\begin{lemma}[Capping-off construction]\label{lem:capping}
Let $L_{\Lambda,\iota}\in \cL_{\Lambda,\iota}$ and assume in addition that the composition
$$H_1(\Lambda;\Q)\xrightarrow{\iota_*} H_1(V\times DT^*S^1;\Q)\xrightarrow{p_*}H_1(DT^*S^1;\Q)$$
is zero. Then $A'_{\min}(\cL_{\Lambda,\iota},\, V\times DT^*S^1)<\infty$.
\end{lemma}
\begin{proof}
Let $(D^2,\lambda_{D^2}=r^2\,\mathrm{d}\theta)$ denote the disk of radius $2$. We have a symplectic embedding
$$\iota_2:(DT^*S^1,\, \mathrm{d}\lambda_{DT^*S^1}) \to (D^2,\, \mathrm{d}\lambda_{D^2}),\qquad (p,\theta) \mapsto (\sqrt{p+2},\,\theta).$$
The difference between $\iota_2^*\lambda_{D^2}$ and $\lambda_{DT^*S^1}$ is
$$\alpha := \iota_2^*\lambda_{D^2}-\lambda_{DT^*S^1}=2\,\mathrm{d}\theta.$$
From $\iota_2$, we obtain a symplectic embedding
$$\iota_1:V\times DT^*S^1 \to V\times D^2,\qquad (x,y)\mapsto (x,\iota_2(y)).$$
In particular, $\iota_1^*(\lambda_{V\times D^2})-\lambda_{V\times DT^*S^1} = 2\,\mathrm{d}\theta$. Since the $S^1$-factor is contractible and $p_*\circ \iota_*=0$, the restriction $\bigl(\iota_1^*(\lambda_{V\times D^2})-\lambda_{V\times DT^*S^1}\bigr)\big|_{\Lambda\times S^1}$ is exact, as one checks by pairing with $H_1(\Lambda\times S^1;\Q)$. Therefore
$$A'_{\min}(\cL_{\Lambda,\iota},\, V\times DT^*S^1)= A'_{\min}(\cL_{\Lambda,\iota_1\circ \iota},\, V\times D^2)\le A'_{\min}(V\times D^2)<\infty,$$
where the last inequality follows from \cite[Theorem 1.3 (1)]{Zho22} together with $SH^*(V\times D^2;\Z)=0$, which holds by \cite{Oan06}.
\end{proof}

\begin{remark}
\cite[Theorem 1.3]{Zho22} states the condition as $SH^*(W;R)=0$ for any ring $R$. However, the local-system issue \cite{Kra18} was overlooked there. In fact, one needs to take $R=\Z$ or $\Z/2$, since $SH^*(T^*M;\Z)$, $SH^*(T^*M;\Z/2)\ne 0$ for any $M$. For a general $R$, one needs to assume in addition that the Lagrangian is spin. 
\end{remark}

\begin{corollary}\label{cor:others}
    If $\Sigma$ is diffeomorphic to $S^2$, $\R P^2$, $T^2$, or the Klein bottle $K^2$, then $A'_{\min}(\cL_{\Sigma,\iota},\, DT^*T^3)<\infty$.
\end{corollary}
\begin{proof}
The image of $\iota_*: H_1(\Sigma;\Z) \to H_1(DT^*T^3;\Z) \cong H_1(T^3;\Z)=\Z^3$ is a subgroup of rank at most $2$, so we can find an isomorphism $\rho\in GL(3,\Z)$ such that $\rho\circ \iota_*$ has image contained in $\Z^2\oplus \{0\}\subset \Z^3$. Treating $\rho$ as a diffeomorphism of $T^3$, we may therefore assume from the outset that $\iota_*$ has this property with respect to the splitting $H_1(T^3;\Z)=H_1(S^1;\Z)\oplus H_1(S^1;\Z)\oplus H_1(S^1;\Z)=\Z^3$. By Lemma~\ref{lem:capping}, we can cap off the last $DT^*S^1$ factor and conclude that $A'_{\min}(\cL_{\Sigma,\iota},\, DT^*T^3)$ is finite, with an upper bound depending only on the isotopy class of $\iota$.
\end{proof}

\begin{proof}[Proof of \Cref{thm:chord}]
Let $\iota_1$ denote the embedding $V\hookrightarrow W$. Since $H_1(\Lambda;\Q)=0$, $A'_{\min}(L_{\Lambda,\iota},V)$ is contributed only by the $S^1$-factor of $\Lambda\times S^1$, and its value is the same whether computed in $V$ or in $W$: the $S^1$-factor is contractible in $V$ (so the pairing equals the symplectic area of a bounding disk), and $W$ is exact (so this symplectic area is independent of the choice of bounding disk). Therefore
$$A'_{\min}(\cL_{\Lambda,\iota},V)\le A'_{\min}(\cL_{\Lambda,\iota_1\circ \iota},W)\le A'_{\min}(W)\le A_{\min}(W)<\infty,$$
where the last bound is provided by \cite[Theorem 1.3 (1)]{Zho22}.

For \eqref{c2}, it suffices to prove that if $\Lambda$ is a closed manifold admitting a metric of negative sectional curvature, then $A'_{\min}(\cL_{\Lambda,\iota},\, V\times DT^*S^1)<\infty$ for any Liouville domain $V$ and any isotopy class $\iota$.
Assume for contradiction that $A'_{\min}(\cL_{\Lambda,\iota},\, V\times DT^*S^1)=\infty$. We may assume $SH^*(V;\Z/2)\ne 0$, for otherwise the result follows directly from \cite[Theorem 1.3 (1)]{Zho22}. On $SH^*(V\times DT^*S^1;\Z/2)\cong SH^*(V;\Z/2)\otimes SH^*(DT^*S^1;\Z/2)$, we have a quasi-dilation $(x,y):=(1\otimes t,\, 1\otimes t\otimes a)$ represented by non-contractible orbits. By Proposition~\ref{prop:BV}, this quasi-dilation $(x,y)$ is visible in $SH^*_{<A}(V\times DT^*S^1;\Z/2)$ for some $A$. Choose $L_{\Lambda,\iota}\in \cL_{\Lambda,\iota}$ such that $A'_{\min}(L_{\Lambda,\iota},\,V\times DT^*S^1)> 4A$. By Proposition~\ref{prop:truncate}, we obtain truncated Viterbo transfers
$$SH^*_{<A\text{ or }2A}(V\times DT^*S^1;\Z/2)\to SH^*(T^*L_{\Lambda,\iota};\Z/2),$$
compatible with the (truncated) BV algebra structure. Consequently, the image of $(x,y)$ is a quasi-dilation in $SH^*(T^*L_{\Lambda,\iota};\Z/2)\cong \H^*(\Lambda;\Z/2)\otimes \H^*(S^1;\Z/2)$. By Proposition~\ref{prop:negaCurv}, $\H^*(\Lambda;\Z/2)$ does not admit a quasi-dilation. Hence every quasi-dilation in $\H^*(\Lambda;\Z/2)\otimes \H^*(S^1;\Z/2)$ must be of the form $(1\otimes t^i,\, 1\otimes t^i\otimes a)$ for an odd integer $i\in \Z$, coming from the $S^1$-factor and represented by a non-trivial homotopy class. However, since the $S^1$-factor of $L_{\Lambda,\iota}$ is contractible in $W$ and the Viterbo transfer respects homotopy classes of generators, the image of $(x,y)$ under the Viterbo transfer cannot be of this form. This is a contradiction.
\end{proof}
\begin{remark}\label{rmk:quasi-dilation}
    The proof of Theorem~\ref{thm:chord}\eqref{c2} still works if $V\times DT^*S^1$ is replaced by any Liouville domain $W$ whose symplectic cohomology admits a quasi-dilation represented by non-contractible orbits.
\end{remark}

We can now assemble the above results to obtain \Cref{thm:main}.
\begin{proof}[Proof of Theorem~\ref{thm:main}]
By Theorems~\ref{thm:chord}\eqref{c2} and~\ref{thm:Mohnke}, the chord conjecture holds for any hyperbolic Legendrian surface in $ST^*T^3$. The low-genus case follows from Corollary~\ref{cor:others} and Theorem~\ref{thm:Mohnke}.
\end{proof}

\subsection{Other examples}
The following is a reformulation of the proof of \cite[Theorem 1.3 (2)]{Zho22}.
\begin{proposition}\label{prop:dia}
Let $W$ be a Liouville domain such that $SH^*(W;R)$ admits a dilation, and let $\Lambda$ be a closed $(\dim_{\C}W-1)$-dimensional manifold such that $SH^*(T^*\Lambda;R)$ admits no dilation. Then $A'_{\min}(\cL_{\Lambda,\iota},\, W)<\infty$ for any $\iota$.
\end{proposition}

\begin{proof}[Proof of Theorem~\ref{thm:CotBundle}]
    We first prove \eqref{cot1}. By Examples~\ref{BVS1}, \ref{BVS2}, \ref{BVRP2} and Remark~\ref{BVT2}, $SH^*(DT^*S^3;\Z/2)$ admits a $\Z/2$-dilation, while $SH^*(DT^*(\Lambda\times S^1);\Z/2)\cong \H^*(\Lambda;\Z/2)\otimes \H^*(S^1;\Z/2)$ does not, for any surface $\Lambda$. Then $A'_{\min}(\cL_{\Lambda,\iota},\, DT^*S^3)<\infty$ by Proposition~\ref{prop:dia}, so the chord conjecture holds.

    For $ST^*(S^1\times S^2)$, the chord conjecture holds for Legendrian surfaces that are not $T^2$ or $K^2$ by Theorem~\ref{thm:chord}. When the Legendrian is $T^2$ or $K^2$, we instead apply Proposition~\ref{prop:dia}, since $SH^*(T^*(S^1\times S^2);\Q)$ admits a dilation by Example~\ref{BVS2}.

    To prove \eqref{cot2}, recall from \cite{Cie02} that $SH^*(V_1\natural V_2)\cong SH^*(V_1)\oplus SH^*(V_2)$ as BV algebras when $\dim V_1\geq 4$, where $V_1\natural V_2$ denotes the boundary connected sum of $V_1$ and $V_2$, i.e.\ $V_1$ and $V_2$ joined by a Weinstein $1$-handle. Therefore, an arbitrary connected sum among the listed collection carries either a dilation or a quasi-dilation represented by non-contractible orbits; thus, the proof of Theorem~\ref{thm:chord}\eqref{c2} implies that the chord conjecture holds for all hyperbolic Legendrians.
    
    It remains to prove the chord conjecture for Legendrians $\Lambda = S^2,\,\R P^2,\,T^2,\,K^2$. For $DT^*(S^1\times \Sigma)$, we can cap off the $DT^*S^1$ factor to $D^2$ as in Lemma~\ref{lem:capping}; this yields a symplectic embedding
    $$\iota_1:DT^*(S^1\times \Sigma)\,\natural\, V \hookrightarrow (D^2\times DT^*\Sigma)\,\natural\, V.$$
    Similarly to the proof of Lemma~\ref{lem:capping}, for an embedding $\iota:\Lambda \to DT^*(S^1\times \Sigma)\,\natural\, V$ such that the composition
    $$\iota_*:H_1(\Lambda;\Q)\to H_1(DT^*(S^1\times \Sigma)\,\natural\, V;\Q)\to H_1(DT^*(S^1\times \Sigma);\Q)\to H_1(DT^*S^1;\Q)$$
    is zero, we have
    $$A'_{\min}(\cL_{\Lambda,\iota},\, DT^*(S^1\times \Sigma)\,\natural\, V) = A'_{\min}(\cL_{\Lambda,\iota_1\circ\iota},\, (D^2\times DT^*\Sigma)\,\natural\, V).$$
    Therefore, if $\Lambda=S^2$ or $\R P^2$, we can cap off all $DT^*(S^1\times \Sigma)$ factors and reduce the consideration of $A'_{\min}$ to a domain of the form $\natural^k DT^*S^3\,\natural\, V$ with $SH^*(V;\Z)=0$. Since $SH^*(DT^*S^3;\Z/2)$ admits a dilation while $SH^*(DT^*S^2;\Z/2)$ and $SH^*(DT^*\R P^2;\Z/2)$ do not, by Examples~\ref{BVS2} and~\ref{BVRP2}, Proposition~\ref{prop:dia} implies that the chord conjecture holds.

    For $\Lambda=T^2$ or $K^2$, in the same spirit as the proof of Corollary~\ref{cor:others}, we can cap off the $DT^*T^3$ components to $D^2\times DT^*T^2$ in a suitable way, and reduce to considering $A'_{\min}(L_{T^2,\iota},W)$ where
    $$W = \natural^{k_1}DT^*S^3\,\natural^{k_2}DT^*(S^1\times S^2)\,\natural^{k_3} DT^*(S^1\times \R P^2)\,\natural\, V$$
    with $SH^*(V;\Q)=0$. Now $W$ admits a dilation over $\Q$ by Examples~\ref{BVS2} and~\ref{BVRP2}; the conclusion then follows from Proposition~\ref{prop:dia}.
\end{proof}

\bibliographystyle{abbrv} 
\bibliography{ref}
\Addresses

\end{document}